\documentclass[a4paper,11pt,reqno]{amsart}
\usepackage[utf8]{inputenc}
\usepackage{amsfonts}
\usepackage{amsmath}
\usepackage{logicproof}
\usepackage{amssymb, amsthm}
\usepackage{tabularx}
\usepackage{booktabs}
\usepackage[labelfont=bf,format=plain,justification=raggedright,singlelinecheck=false]{caption}
\usepackage{mathrsfs}
\usepackage[]{mdframed}
\usepackage[colorlinks]{hyperref}
\usepackage{tcolorbox}
\usepackage{verbatim}
\tcbuselibrary{theorems}
\newtcbtheorem[number within=section]{mytheo}{Note}%
{colback=white!5,colframe=black!35!black,fonttitle=\bfseries}{th}
\numberwithin{equation}{section}
\usepackage{graphicx}
\usepackage{textgreek}

\newtheorem{theorem}{Theorem}[section]

\newtheorem{lemma}[theorem]{Lemma}

\begin{document}
	\title[Kernel Conjugate Gradient for Functional Linear Regression]{Convergence Analysis of Kernel Conjugate Gradient for Functional Linear Regression}
\author[N. Gupta]{Naveen Gupta}
\address[N. Gupta]{Indian Institute of Technology Delhi, India}
\email{ngupta.maths@gmail.com}
\author{S. Sivananthan}
\address[S. Sivananthan]{Indian Institute of Technology Delhi, India}
\email{siva@maths.iitd.ac.in}
\author[B. K. Sriperumbudur]{Bharath K. Sriperumbudur}
\address[B. K. Sriperumbudur]{Pennsylvania State University, USA}
\email{bks18@psu.edu}

\begin{abstract}
	In this paper, we discuss the convergence analysis of the conjugate gradient-based algorithm for the functional linear model in the reproducing kernel Hilbert space framework, utilizing early stopping results in regularization against over-fitting. We establish the convergence rates depending on the regularity condition of the slope function and the decay rate of the eigenvalues of the operator composition of covariance and kernel operator. Our convergence rates match the minimax rate available from the literature.
\end{abstract}
\maketitle

 
 \section{Introduction}
\noindent
The functional linear regression (FLR) model is one of the fundamental tools for analyzing functional data, introduced by Ramsay and Dalzell \cite{ramsay1991some}. The model gained popularity due to its simplicity in dealing with high-dimensional functional data.
For example, it is widely used in medicine, chemometrics, and economics \cite{ramsay2005FDA,ferraty2006nonparametric,forni1998let,preda2005clusterwise}. Mathematically, the FLR model is stated as 
\begin{equation*}
\label{model}
    Y = \int_{S} X(t) \beta^*(t)\, dt + \epsilon,
\end{equation*}
where $Y$ is a real-valued random variable, $\left(X(t); t\in S \right)$ is a continuous time process, $\beta^*$ is an unknown slope function and $\epsilon$ is a zero mean random noise, independent of $X$, with finite variance $\sigma^2$. Throughout the paper, we assume that $X$ and $\beta^*$ are in $L^2\left(S\right)$, and $S$ is a compact subset of $\mathbb{R}^d$. 
In the context of the slope function, it is evident that
\begin{equation*}
\label{true solution}
    \beta^* := \arg\min_{\beta \in L^2\left(S\right)}\mathbb{E}\left[Y-\langle X,\beta \rangle \right]^2.
\end{equation*}
The goal is to construct an estimator $\hat{\beta}$ to approximate the slope function $\beta^*$ using observed empirical data 
$\left\{\left(X_1,Y_1\right),\left(X_2,Y_2\right),\cdots, \left(X_n,Y_n\right)\right\}$, where $X_i$'s are i.i.d.~copies of random function $X$.
The main approach in estimating the slope function $\beta^*$ is based on the representation of the estimator function and functional data in terms of certain basis functions. In this paper, we utilize the framework of reproducing kernel Hilbert space (RKHS) to construct an estimator function $\hat{\beta}$ using the conjugate gradient method.

In \cite{cardot2003spline}, the authors used penalized B-spline basis functions to represent the estimator function and also introduced an alternate smooth version of functional principal component analysis (FPCA) to construct $\hat{\beta}$. A Fourier basis approach was explored in \cite{li2007rates} and the FPCA-based approach is investigated in \cite{cai2006prediction,hall2007methodology,muller2005generalized}. One of the profound choices for the basis functions in the FPCA method is to use the eigenfunctions obtained from the covariance operator of the given data. However, Cai and Yuan~\cite{tonyyuan2012minimax} demonstrated that this choice may not be suitable for all cases, as shown with the example of Canadian weather data. This observation strongly motivated the researchers to explore alternative choices of basis functions. 

It is well-known in learning theory that kernel methods represent predictor functions using data-driven kernel functions, resulting in good generalization error (see \cite{smale2002foundation,sergeionregularization,cuckerzhou2007learningtheory,AI2022sergei}). 
Cai and Yuan \cite{ARKHSFORFLR} proposed utilizing the kernel method approach, wherein the estimator is expressed as a linear combination of kernel functions. The method achieves optimal rates under the assumption that the slope function $\beta^*$ belongs to the RKHS. Later in \cite{tonyyuan2012minimax}, they used the regularization technique to achieve optimal rates without the Sacks–Ylvisaker condition, which was a necessary assumption in \cite{ARKHSFORFLR}. Further analysis of the FLR model within the framework of RKHS has been studied and discussed in \cite{balasubramanian2022unified,ZhangFaster2020,HTONGFLR,TONGHUBER}. Since the computational complexity of these techniques is $O(n^3)$, they incur high computational costs when dealing with large datasets.

To address this shortcoming, Blanchard and Kr\"amer\cite{BlanchardCGKCG2016} employed the conjugate gradient method in the kernel ridge regression method, by utilizing an early stopping rule that also serves as a form of regularization. This reduces the computational complexity to $O\left(n^2m\right)$ for $m$ number of iterations \cite{KCGMwithrandomprojection}. Inspired by their work, in this paper, we propose an estimator $\hat{\beta}$ for the FLR model by employing the conjugate gradient (CG) approach with an early stopping rule. We specifically focus on the CG method due to its outstanding computational characteristics, setting it apart from other approaches. Since it aggressively targets the reduction of residual errors, it is commonly observed in practical applications that the CG method achieves convergence in significantly fewer iterations compared to other gradient descent techniques, as discussed in the context of kernel learning by \cite{guo2023capacity} and \cite{onlineGD2022}. We obtain a convergence rate for $\Vert \hat{\beta}-\beta^*\Vert_{L^2(S)}$ and show it to align with the minimax rates of the FLR model \cite{tonyyuan2012minimax,ZhangFaster2020}, thereby establishing the minimax optimality of our estimator.

The paper is organized as follows.
In Section~\ref{preliminaries}, we present the necessary background for the conjugate gradient method for the FLR model in the RKHS setting and explain some important properties of certain orthogonal polynomials. In Section~\ref{mainresult}, we discuss our assumptions and provide convergence rates of the CG method. We present the supplementary results, which will be used to prove the main theorem, in the final section, Section \ref{supplements}.
\section{Preliminaries and Notations}
\label{preliminaries}
\noindent Let $\mathcal{H}$ be  a Hilbert space of real-valued functions on a compact subset $S$ of $\mathbb{R}^d$.
We say that $\mathcal{H}$ is  RKHS if for every $x\in S$, the pointwise evaluation map $f \mapsto f(x)  $ is continuous on $\mathcal{H}$. As a consequence of the Riesz representation theorem, there is a unique kernel function $k:S\times S\rightarrow\mathbb{R}$, called the \emph{reproducing kernel} such that $k\left(s,\cdot\right) \in \mathcal{H}$ for any $s \in S$ satisfies the reproducing property:
\begin{equation*}
f\left(s\right) = \left\langle k(s,\cdot) ,f \right\rangle_{\mathcal{H}}, \quad \forall f \in \mathcal{H}.
\end{equation*}
It is easy to see that the associated kernel function $k$ is symmetric and positive definite. Conversely, for a given symmetric and positive definite function $k$, we can construct a unique RKHS with $k$ as the reproducing kernel.  For a detailed study of RKHS, we
refer the reader to \cite{aronszajn1950rkhs}.


We assume that $k$ is continuous;  
 then the associated RKHS $\mathcal{H}$ is separable and the embedding operator (inclusion operator) $J: \mathcal{H} \to L^2\left(S\right)$, which is defined as $(Jf)(x) = \left\langle k(x,\cdot) ,f \right\rangle_{\mathcal{H}}$ is compact. The adjoint operator $J^*:L^2\left(S\right) \to \mathcal{H}$ is given by 
$$(J^*g)(x)= \int_S k(x,t)g(t)\, dt.$$  We denote the integral operator, $T:= JJ^{*}: L^2\left(S\right) \to L^2\left(S\right)$ and
the covariance operator $C:= \mathbb{E}\left[X \otimes X\right]: L^2\left(S\right) \to L^2\left(S\right)$, 
where $\otimes$ is the $L^2\left(S\right)$ tensor product.

Given $(X_i,Y_i)_{i=1}^n$ i.i.d. copies of $\left(X,Y\right)$, our estimator is defined as
\begin{equation*}
    \begin{split}
        \hat{\beta} & = \arg\min_{\beta \in \mathcal{H}}\frac{1}{n}\sum_{i=1}^n\left[Y_i-\langle \beta, X_i \rangle_{L^2} \right]^2\\
        & = \arg\min_{\beta \in \mathcal{H}}\frac{1}{n}\sum_{i=1}^n\left[Y_i-\langle J\beta, X_i \rangle_{L^2} \right]^2\\
        & = \arg\min_{\beta \in \mathcal{H}}\frac{1}{n}\sum_{i=1}^n\left[Y_i-\langle \beta, J^{*}X_i \rangle_\mathcal{H} \right]^2.\\
    \end{split}
\end{equation*}
A solution for this optimization problem can be obtained by solving 
\begin{equation}
\label{estimator}
J^{*}\hat{C}_{n}J\hat{\beta} = J^{*}\hat{R},
\end{equation}
where 
\begin{equation*}
    \hat{C}_{n}:= \frac{1}{n}\sum_{i=1}^n X_i \otimes X_i \quad \text{and} \quad \hat{R}:= \frac{1}{n}\sum_{i=1}^n Y_i X_i.
\end{equation*}
We denote $\Lambda:= T^{\frac{1}{2}}CT^{\frac{1}{2}}$ and $\Lambda_{n}:= T^{\frac{1}{2}}\hat{C}_{n}T^{\frac{1}{2}}$. 

The fundamental idea behind the conjugate gradient method is to restrict the optimization problem to a set of subspaces (data dependent), known as Krylov subspaces, defined as
\begin{equation*}
\begin{split}
    \mathcal{K}_{m}\left(J^{*}\hat{R}, J^{*} \hat{C}_{n}J\right) : &= \text{span}\left\{J^{*}\hat{R}, J^{*}\hat{C}_{n}J J^{*}\hat{R}, \left(J^{*}\hat{C}_{n}J\right)^2J^{*}\hat{R},\ldots,  \left(J^{*}\hat{C}_{n}J\right)^{m-1}J^{*}\hat{R} \right\}\\
    & = \left\{p\left(J^{*} \hat{C}_{n}J\right)J^{*}\hat{R} \quad : \quad p \in \mathcal{P}_{m-1}\right\},
\end{split}
\end{equation*}
where $\mathcal{P}_{m-1}$ is a set of real polynomials of degree at most $m-1$. Then the CG solution after $m$ iterations is 
\begin{equation}
\label{kryestimator}
    \hat{\beta}_{m} = \arg\min_{\beta \in \mathcal{K}_{m}\left(J^{*}\hat{R}, J^{*} \hat{C}_{n}J\right)}\left\|J^{*}\hat{R}- J^{*} \hat{C}_{n}J \beta\right\|_\mathcal{H}.
\end{equation}
The iterated solution, because the problem is restricted to the Krylov subspaces, will take the form $\hat{\beta}_{m} = q_{m}\left(J^{*}\hat{C}_{n}J\right)J^{*}\hat{R}$ with $q_{m}$ being a polynomial of degree at most $m-1$. Associated with each iterated polynomial $q_m$, we have a residual polynomial defined as $p_m\left(x\right) = 1-xq_m\left(x\right) \in \mathcal{P}_{m}^0$, where $\mathcal{P}_{m}^0$ is a set of real polynomials of degree at most $m$ and having constant term equal to 1. 
\noindent
Since the construction of the estimator involves forward multiplication through the residual polynomial $p_m$, it is essential to understand certain fundamental characteristics of these polynomials.  

Suppose $\left(\xi_{n,i},e_{n,i}\right)_{i\in I}$  is an eigenvalue-eigenfunction pair for the operator $\hat{\Lambda}_{n}$ with $\xi_{n,i}$ in $ \left[0,\kappa \right]$, $i\in I$, for some $\kappa >0$  and $\left\{e_{n,i} : i\in I\right\}$ is an orthonormal system in $L^2\left(S\right)$.
For $u>0$, denote $F_{u}= 1_{\left[0,u\right)}\left(\hat{\Lambda}_{n}\right)$ as the orthogonal projector onto the space spanned by $\left\{e_{n,i} : i\in I,~  \xi_{n,i} < u\right\}$.
For each integer $l \geq 0$, we will introduce measure $\mu_{n}^{\left(l\right)}$ which is defined as
\begin{equation*}
    \mu_{n}^{\left(l\right)}:= \sum_{i \in I} \xi_{n,i}^l \left\langle T^{\frac{1}{2}}\hat{R},e_{n,i}\right\rangle_{L^2}^2 \delta_{\xi_{n,i}},
\end{equation*}
where $\delta_{x}$ is Dirac measure centered at $x$. In particular, for $l=0$ we use the convention $0^0=1$.
Associated to each measure $\mu_{n}^{\left(l\right)}$, $l\geq 0$, we define the scalar product of two polynomials as 
\begin{equation*}
    \begin{split}
        \left[p,q\right]_{\left(l\right)}: = & \int_{0}^{\kappa}p\left(t\right)q\left(t\right)d\mu_{n}^{\left(l\right)}\left(t \right) \\
         = & \left\langle p\left(\hat{\Lambda}_{n}\right)T^{\frac{1}{2}}\hat{R}, \left(\hat{\Lambda}_{n}\right)^{l} q\left(\hat{\Lambda}_{n}\right)T^{\frac{1}{2}} \hat{R} \right\rangle_{L^2}  \\
         = & \sum_{i \in I } p\left(\xi_{n,i}\right)  q\left(\xi_{n,i}\right) \left(\xi_{n,i}\right)^l \left\langle T^{\frac{1}{2}}\hat{R},e_{n,i}\right\rangle_{L^2}^2.
    \end{split}
\end{equation*}
For $l=0$, we see that
\begin{equation*}
    \left[p,q\right]_{\left(0\right)}= \left\langle p\left(\hat{\Lambda}_{n}\right)T^{\frac{1}{2}}\hat{R}, q\left(\hat{\Lambda}_{n}\right)T^{\frac{1}{2}} \hat{R} \right\rangle_{L^2}.
\end{equation*}
Since  $\hat{\Lambda}_{n}$ is a finite rank operator, it has only a finite number of non-zero eigenvalues. Consequently, we observe that the measure $\mu_n^{(l)}$ has finite support of cardinality, independent of $l$. Indeed, if
 $\xi_{n,j}=0$ for some $j \in I$, then we have
\begin{equation*}
    \begin{split}
    T^{\frac{1}{2}}\hat{C}_{n}T^{\frac{1}{2}} e_{n,j}=0 \implies
    & \left \langle T^{\frac{1}{2}}\hat{C}_{n}T^{\frac{1}{2}} e_{n,j} , e_{n,j} \right \rangle_{L^2} = 0 \\
    \implies & \frac{1}{n}\sum_{i=1}^n \left \langle T^{\frac{1}{2}} X_{i}, e_{n,j} \right \rangle_{L^2}^2 = 0 \\
    \implies & \left \langle T^{\frac{1}{2}} X_{i}, e_{n,j} \right \rangle_{L^2} = 0, \quad \forall\,1\leq i \leq n.
    \end{split}
\end{equation*}
Now we see that
\begin{equation*}
    \begin{split}
        \left \langle T^{\frac{1}{2}} \hat{R}, e_{n,j} \right \rangle_{L^2}  & = \frac{1}{n}\sum_{i=1}^n Y_{i} \left \langle T^{\frac{1}{2}} X_{i}, e_{n,j} \right \rangle_{L^2}= 0.
    \end{split}
\end{equation*}
This concludes that $\mu_{n}^{\left(l\right)}$ has finite support of cardinality (independent of $l$), let's say $n_{\gamma} \leq n$. It is clear from $\left(\ref{kryestimator}\right)$ that $q_{m}$ minimizes 
\begin{equation*}
\left\|J^{*}\hat{R}-J^{*}\hat{C}_{n}Jq\left(J^{*}\hat{C}_{n}J\right)J^{*}\hat{R}\right\|_{\mathcal{H}}= \left\|T^{\frac{1}{2}}\hat{R}-T^{\frac{1}{2}}\hat{C}_{n}T^{\frac{1}{2}}q\left(T^{\frac{1}{2}}\hat{C}_{n}T^{\frac{1}{2}}\right)T^{\frac{1}{2}}\hat{R}\right\|_{L^2}
\end{equation*}
over $q \in \mathcal{P}_{m-1}$. Equivalently, consider $p\left(x\right)=1-xq\left(x\right)$, then  $p_{m}$  minimizes $$\left\|p\left(T^{\frac{1}{2}}\hat{C}_{n}T^{\frac{1}{2}}\right)T^{\frac{1}{2}}\hat{R}\right\|_{L^2}= \left[p,p\right]_{\left(0\right)}$$ over $p \in \mathcal{P}_{m}^0$. In other words, we can say that $p_{m}$ is the orthogonal projection of origin onto the affine subspace $\mathcal{P}_{m}^{0} \subset \mathcal{P}_{m}$ for the scalar product $\left[\cdot,\cdot\right]_{\left(0\right)}$. We take $q_{0}=0, p_{0}=1$ for $m=0$. Because of the properties of projections, $p_{m}$ is orthogonal to $\mathcal{P}_{m}^{0} $. Since $\mathcal{P}_{m}^{0} = 1+ \pi \mathcal{P}_{m-1}$ is parallel to $\pi \mathcal{P}_{m-1}$, where $\left(\pi(q)\right)(x) = xq(x)$ is a shift operator. So we have  $0=\left[p_{m},\pi q\right]_{\left(0\right)}=\left[p_{m},q\right]_{\left(1\right)}$ for any $q \in \mathcal{P}_{m-1}$ which concludes that $p_{0},\ldots,p_{n_{\gamma-1}}$ is an orthogonal sequence of polynomials with respect to $\left[\cdot,\cdot\right]_{\left(1\right)}$. For $m=n_{\gamma}$, we can see that $\left[\cdot,\cdot \right]_{\left(l\right)}$ is semidefinite product on $\mathcal{P}_{n_{\gamma}}$ and $p_{n_{\gamma}}$ is unique element of $\mathcal{P}_{n_{\gamma}}^0$ satisfying
$\left[p_{n_{\gamma}},p_{n_{\gamma}} \right]_{\left(0\right)}=0 $. Hence, for $m=n_{\gamma}$, unicity of the solution holds and $\left[p_{n_{\gamma}},p_{m}\right]_{\left(1\right)}=0$ for all $m \leq n_{\gamma}$. By applying the representer theorem to $\left(\ref{estimator}\right)$,
\begin{equation*}
    \hat{\beta} \in \text{span}\left\{\int_{S}k(\cdot,t)X_{i}(t)\,dt : i=1,\ldots,n\right\},
\end{equation*}
i,e., there exists a $\mathbf{\alpha}:=\left(\alpha_{1},\ldots,\alpha_{n}\right)^{\top} \in \mathbb{R}^{n}$ such that $\hat{\beta}=\sum_{i=1}^{n}\alpha_{i}\int_{S}k\left(\cdot,t\right)X_{i}\left(t\right)dt$. Using this in $\left(\ref{estimator}\right)$, we can solve $\mathbf{K \alpha}=\mathbf{y}$ to get a solution of $\left(\ref{estimator}\right)$, where
\begin{equation*}
    \mathbf{K} \in \mathbb{R}^{n \times n} \text{ with } \left[\mathbf{K}\right]_{ij:}= \int_{S}\int_{S}k(s,t)X_{i}(t)X_{i}(s)\,dt\,ds
\end{equation*} 
and $\mathbf{y}= \left(Y_{1},\ldots,Y_{n}\right)^{\top}\in \mathbb{R}^{n}$. We refer the reader to  \cite{Hanke1995CGtypemethod} for the iterative methodology to create the CG estimator of $\hat{\beta}$ in $\left(\ref{estimator}\right)$.

The following lemma, which lists several properties of orthogonal polynomials, is proven in \cite{BlanchardCGKCG2016}. It will be used frequently throughout the remainder of the paper.
\begin{lemma}\cite{BlanchardCGKCG2016}
\label{polynomial lemma}
    Let $m$ be any integer satisfying $1\leq m \leq n_{\gamma}$.
    \begin{enumerate}
        \item The polynomial $p_{m}$ has exactly $m$ distinct roots belonging to $\left(0,\kappa \right]$, denoted by $\left(x_{k,m}\right)_{1\leq k \leq m}$ in increasing order.
        \item $p_m$ is positive, decreasing and convex on the interval $\left[0,x_{1,m}\right)$.
        \item Define the function $\phi_{m}$ on the interval $\left[0,x_{1,m}\right)$ as
        \begin{equation*}
            \phi_{m}\left(x\right) = p_{m}\left(x\right)\left(\frac{x_{1,m}}{x_{1,m}-x}\right)^{\frac{1}{2}}.
        \end{equation*}
        Then it holds 
        \begin{equation*}
        \begin{split}    \left[p_m,p_m\right]_{\left(0\right)}^{\frac{1}{2}} & = \left\|p_{m}\left(\hat{\Lambda}_{n}\right)T^{\frac{1}{2}}\hat{R}\right\|_{L^2}
             \leq \left\|F_{x_{1,m}}\phi_{m}\left(\hat{\Lambda}_{n}\right)T^{\frac{1}{2}}\hat{R}\right\|_{L^2},
        \end{split}
        \end{equation*}
        and furthermore, for any $\nu \geq 0$, 
        \begin{equation}
        \label{xphi bound}
            \sup_{x \in \left[0,x_{1,m}\right]}x^{\nu}\phi_{m}^2\left(x\right) \leq \nu^{\nu} \left|p_{m}^{'}\left(0\right)\right|^{-\nu}.
        \end{equation}
        \item Denote $p_{0}^{\left(2\right)},p_{1}^{\left(2\right)},\ldots,p_{n_{\gamma-1}}^{\left(2\right)}$ the unique sequence of orthogonal polynomial with respect to $\left[\cdot,\cdot\right]_2$ and with constant term equal to 1. 
        This sequence enjoys properties $\left(1\right)$ and $\left(2\right)$ above with $\left(x_{k,m}^{\left(2\right)}\right)_{1\leq k\leq m}$ denoting the distinct roots of $p_{m}^{\left(2\right)}$ in increasing order. Then it holds that $x_{1,m} \leq x_{1,m}^{\left(2\right)}$. 
        Finally, the following holds:
        \begin{equation*}
            0 \leq p_{m-1}^{'}\left(0\right) - p_{m}^{'}\left(0\right) = \frac{\left[p_{m-1},p_{m-1}\right]_{\left(0\right)}-\left[p_{m},p_{m-1}\right]_{\left(0\right)}}{\left[p_{m-1}^{\left(2\right)},p_{m-1}^{\left(2\right)}\right]_{\left(1\right)}} \leq \frac{\left[p_{m-1},p_{m-1}\right]_{\left(0\right)}}{\left[p_{m-1}^{\left(2\right)},p_{m-1}^{\left(2\right)}\right]_{\left(1\right)}}.
        \end{equation*}
    \end{enumerate}
\end{lemma}

\section{The Main Result}
\label{mainresult}
\noindent
In this section, we present the convergence rate of the conjugate gradient method in functional linear regression. Our proofs are inspired by the ideas of  \cite{BlanchardCGKCG2016}. The analysis depends on the eigenvalue behaviour of the operator $\Lambda = T^{\frac{1}{2}}CT^{\frac{1}{2}}$ which indicates the behaviour of eigenvalues of the kernel operator $T$ and the covariance operator $C$.
First, we begin with a list of assumptions that are required for our convergence rate analysis.

\noindent \textbf{Assumption 1.} $\left( \textit{Source Condition}\right)$ There exists $g \in L^2$ such that $\beta^* = T^{\frac{1}{2}} \left(T^{\frac{1}{2}} C T^{\frac{1}{2}}\right)^{\alpha} g$, where $\alpha$ is any positive real number.

Note that the assumption implies $\beta^* \in \mathcal{H} $ with additional smoothness. In \cite{ZhangFaster2020}, the authors use this source condition to derive the minimax and faster convergence rate for the Tikhonov regularization with $0 < \alpha \leq \frac{1}{2}$.\vspace{1mm}

\noindent \textbf{Assumption 2.} $\left( \textit{Decay Condition}\right)$ For some $s\in(0,1)$, \begin{equation*}
    \label{decay condition}
    i^{-\frac{1}{s}} \lesssim \xi_{i} \lesssim i^{-\frac{1}{s}} \quad \forall \quad i \in I,
\end{equation*}
where $\left(\xi_{i},e_{i}\right)_{i\in I}$ is the eigenvalue-eigenvector pair of operator $\Lambda$ and the symbol $\lesssim$ means that there exist constants $b,B>0$ such that $b i^{-\frac{1}{s}} \leq \xi_{i} \leq B i^{-\frac{1}{s}}$ for all $i \in I$. \vspace{1mm}

 This decay of eigenvalues is related to the effective dimensionality because under this assumption we have that $\mathcal{N}(\lambda):= \text{trace}(\Lambda\left(\Lambda+\lambda I\right)^{-1}) \leq D^2 \left(\kappa \lambda\right)^{-s} \text{ for all } \lambda \in \left(0,1\right]$ and an appropriate choice of $D>0$.\vspace{1mm}
 
\noindent\textbf{Assumption 3.} $\left( \textit{Fourth Moment Condition}\right)$ For any $f \in L^2\left(S\right)$,
\begin{equation*}   
\mathbb{E}\left\langle X,f\right \rangle_{L^2}^{4} \leq c_0 \left(\mathbb{E}\left\langle X,f\right \rangle_{L^2}^{2}\right)^2 \text{ for some constant } c_0>0.
\end{equation*}

We define the early stopping rule to stop the CG method at an early stage $m^* \ll n$. This is mainly used for its implicit regularization property. The early stopping, defined as $m^*$, is the first iteration for which the residual term is less than some predefined threshold. Now we state and prove our main theorem.

\begin{theorem}
    Let $\alpha > 0$, $\tau >0$ and $\mathbb{E}\left\|X\right\|^4 <\infty$. Suppose Assumptions $\left(1\right)$--$\left(3\right)$ hold, and stopping rule holds with threshold $$\Omega = \left(2+\tau\right) n^{-\frac{\alpha+1}{1+s+2\alpha}}.$$
    Then for large enough $n$ and $ \lambda = c\left(\alpha,\delta\right) n^{-\frac{1}{1+s+2\alpha}}$, it holds with probability at least $1-\delta $ that
\begin{equation*}
\left\| \hat{\beta}_{m^*} - \beta^* \right\|_{L^2} \lesssim n^{-\frac{\alpha}{1+s+2 \alpha}},
\end{equation*}
where $c\left(\alpha,\delta\right)$ is a constant that depends only on $\alpha$ and $\delta$.
\end{theorem}
\begin{proof}
    Let $\lambda \geq \left(\frac{4c_{1}^2}{n}\right)^{\frac{1}{s+1}}$ and define $F_{u}^{\perp}:= \left(I-F_{u}\right)$. We start by considering the error term:
    \begin{equation*}
    \begin{split}
        \left\| \hat{\beta}_{m} - \beta^* \right\|_{L^2}  = & \left\| J \hat{\beta}_{m} -\beta^* \right\|_{L^2} \\
         = & \left\|Jq_{m}\left(J^{*}\hat{C}_{n}J\right)J^{*}\hat{R}-\beta^*\right\|_{L^2}\\
         = & \left\|T^{\frac{1}{2}}q_{m}\left(\hat{\Lambda}_{n}\right)T^{\frac{1}{2}}\hat{R}-T^{\frac{1}{2}}\Lambda^{\alpha} g\right\|_{L^2}\\
         \leq & \left\|T^{\frac{1}{2}}\right\|_{op}\left\|q_{m}\left(\hat{\Lambda}_{n}\right)T^{\frac{1}{2}}\hat{R}-\Lambda^{\alpha} g\right\|_{L^2}\\
        \lesssim & \left\|q_{m}\left(\hat{\Lambda}_{n}\right)T^{\frac{1}{2}}\hat{R}-\Lambda^{\alpha} g\right\|_{L^2}\\
        \leq & \left\|F_{u}\left(q_{m}\left(\hat{\Lambda}_{n}\right)T^{\frac{1}{2}}\hat{R}-\Lambda^{\alpha} g\right)\right\|_{L^2} + \left\|F_{u}^{\perp}\left(q_{m}\left(\hat{\Lambda}_{n}\right)T^{\frac{1}{2}}\hat{R}-\Lambda^{\alpha} g\right)\right\|_{L^2}\\
         \leq &\underbrace{\left\|F_{u}\left(q_{m}\left(\hat{\Lambda}_{n}\right)T^{\frac{1}{2}}\hat{R}-q_{m}\left(\hat{\Lambda}_{n}\right)T^{\frac{1}{2}}\hat{C}_{n}\beta^*\right)\right\|_{L^2}}_{Term-1}\\
         &\qquad+\underbrace{\left\|F_{u}\left(q_{m}\left(\hat{\Lambda}_{n}\right)T^{\frac{1}{2}}\hat{C}_{n}\beta^*-\Lambda^{\alpha} g \right)\right\|_{L^2}}_{Term-2} \\
         & \qquad\qquad+ \underbrace{\left\|F_{u}^{\perp}\left(q_{m}\left(\hat{\Lambda}_{n}\right)T^{\frac{1}{2}}\hat{R}-\Lambda^{\alpha} g\right)\right\|_{L^2}}_{Term-3}.
    \end{split}
    \end{equation*}
In the third step we have used that $Jq_{m}\left(J^*\hat{C}_{n}J\right)J^* = T^{\frac{1}{2}}q_{m}\left(T^{\frac{1}{2}}\hat{C}_{n}J\right)T^{\frac{1}{2}}$ which can be derived easily using spectral representation. Further, we will estimate each term separately using Lemma \ref{Covariance estimation} and Lemma \ref{Empirical bound}.\vspace{2mm}\\
\textit{Estimation of Term-1:} 
\begin{equation*}
    \begin{split}
     &  \left\|F_{u}\left(q_{m}\left(\hat{\Lambda}_{n}\right)T^{\frac{1}{2}}\left(\hat{R}-\hat{C}_{n}\beta^*\right)\right)\right\|_{L^2}\\
     \leq & \left\|F_{u}q_{m}\left(\hat{\Lambda}_{n}\right)\left(\hat{\Lambda}_{n}+\lambda I\right)^{\frac{1}{2}}\right\|_{op} \left\|\left(\hat{\Lambda}_{n}+\lambda I\right)^{-\frac{1}{2}}\left(\Lambda+\lambda I\right)^{\frac{1}{2}}\right\|_{op} \\
     & \times \left\|\left(\Lambda+\lambda I\right)^{-\frac{1}{2}}T^{\frac{1}{2}}\left(\hat{R}-\hat{C}_{n}\beta^*\right)\right\|_{L^2} \\
     \lesssim & \sqrt{\frac{ \sigma^2 \mathcal{N}\left(\lambda\right)}{n \delta}}\left\|F_{u}q_{m}\left(\hat{\Lambda}_{n}\right)\left(\hat{\Lambda}_{n}+\lambda I\right)^{\frac{1}{2}}\right\|_{op} \\
     \leq &  \sqrt{\frac{ \sigma^2 \mathcal{N}\left(\lambda\right)}{n \delta}} \left(\sup_{x\in \left[0,u \right]}x^{\frac{1}{2}}q_{m}\left(x\right)+\lambda^{\frac{1}{2}}\sup_{x\in \left[0,u \right]}q_{m}\left(x\right)\right)\\
    \leq &  \sqrt{\frac{ \sigma^2 \mathcal{N}\left(\lambda\right)}{n \delta}} \left(\left(\sup_{x\in \left[0,u \right]}q_{m}\left(x\right)\right)^{\frac{1}{2}}\left(\sup_{x\in \left[0,u \right]} x q_{m}\left(x\right)\right)^{\frac{1}{2}} + \lambda^{\frac{1}{2}}\left|p_{m}^{'}\left(0\right)\right|\right) \\
    \leq &  \sqrt{\frac{ \sigma^2 \mathcal{N}\left(\lambda\right)}{n \delta}} \left(\left|p_{m}^{'}\left(0\right)\right|^{\frac{1}{2}}+\lambda^{\frac{1}{2}}\left|p_{m}^{'}\left(0\right)\right|\right).
    \end{split}
\end{equation*}
 For the last inequality we use that $u \leq x_{1,m}$, $p_{m}$ is decreasing and convex in $\left[0,u\right]$ and $q_{m}\left(x\right)\leq \left|p_{m}^{'}\left(0\right)\right|$ for all $x \in \left[0,u\right)$\\
\\
\textit{Estimation of Term-2:} Using Lemma \ref{fb}, and the fact that $\left|p_{m}\left(x\right)\right| \leq 1$ for all $x \in \left[0,u\right)$, we get
\begin{equation*}
    \begin{split}
        & \left\|F_{u}\left(q_{m}\left(\hat{\Lambda}_{n}\right)T^{\frac{1}{2}}\hat{C}_{n}\beta^*-\Lambda^{\alpha} g \right)\right\|_{L^2} \\
        = & \left\|F_{u}\left(q_{m}\left(\hat{\Lambda}_{n}\right)T^{\frac{1}{2}}\hat{C}_{n}T^{\frac{1}{2}}\Lambda^{\alpha}-\Lambda^{\alpha} g \right) \right\|_{L^2}\\
        = & \left\|F_{u}\left(q_{m}\left(\hat{\Lambda}_{n}\right)\hat{\Lambda}_{n}-I\right)\Lambda^{\alpha} g\right\|_{L^2} 
        = \left\|F_{u}p_{m}\left(\hat{\Lambda}_{n}\right)\Lambda^{\alpha} g\right\|_{L^2} \\
        \leq & 2 \left(\sup_{t \in \left[0,u\right]}t^{\alpha}p_{m}\left(t\right)+ \max\left\{\alpha, 1 \right\} Z_{\alpha}\left(\lambda\right)\sup_{t \in \left[0,u\right]}p_{m}\left(t\right)\right)\\
        \leq & 2 \left({u}^{\alpha}+ \max\left\{\alpha, 1 \right\}Z_{\alpha}\left(\lambda\right)\right).
    \end{split}\vspace{1mm}
\end{equation*}
\textit{Estimation of Term-3:}
\begin{equation*}
    \begin{split}
        & \left\|F_{u}^{\perp}\left(q_{m}\left(\hat{\Lambda}_{n}\right)T^{\frac{1}{2}}\hat{R}-\Lambda^{\alpha} g\right)\right\|_{L^2}\\
        = & \left\|F_{u}^{\perp}\left(\hat{\Lambda}_{n}\right)^{-1}\hat{\Lambda}_{n}\left(q_{m}\left(\hat{\Lambda}_{n}\right)T^{\frac{1}{2}}\hat{R}-\Lambda^{\alpha} g\right)\right\|_{L^2}\\
        \leq & \left\|F_{u}^{\perp}\left(\hat{\Lambda}_{n}\right)^{-1}\left(\hat{\Lambda}_{n}q_{m}\left(\hat{\Lambda}_{n}\right)T^{\frac{1}{2}}\hat{R}-T^{\frac{1}{2}}\hat{R}+T^{\frac{1}{2}}\hat{R}-\hat{\Lambda}_{n}\Lambda^{\alpha} g\right)\right\|_{L^2}\\
        \leq & \left\|F_{u}^{\perp}\left(\hat{\Lambda}_{n}\right)^{-1}\left(\hat{\Lambda}_{n}q_{m}\left(\hat{\Lambda}_{n}\right)T^{\frac{1}{2}}\hat{R}-T^{\frac{1}{2}}\hat{R}\right)\right\|_{L^2}+\left\|F_{u}^{\perp}\left(\hat{\Lambda}_{n}\right)^{-1}\left(T^{\frac{1}{2}}\hat{R}-\hat{\Lambda}_{n}\Lambda^{\alpha} g\right)\right\|_{L^2}
        \end{split}
        \end{equation*}
        \begin{equation*}
        \begin{split}
        \leq & \frac{1}{u}\left\|\hat{\Lambda}_{n}q_{m}\left(\hat{\Lambda}_{n}\right)T^{\frac{1}{2}}\hat{R}-T^{\frac{1}{2}}\hat{R}\right\|_{L^2} + \frac{\left(u+\lambda\right)^{\frac{1}{2}}}{u}\sqrt{\frac{2 \sigma^2 \mathcal{N}\left(\lambda\right)}{n \delta}}.
    \end{split}
\end{equation*}
So by combining all three terms, we get
\begin{equation*}
\label{error bound}
    \begin{split}
        \left\|\hat{\beta}_{m}-\beta^*\right\|_{L^2} \lesssim & \sqrt{\frac{ \sigma^2 \mathcal{N}\left(\lambda\right)}{n \delta}}\frac{\left(\tilde{u}+\lambda\right)^{\frac{1}{2}}}{\tilde{u}}+ \left({u}^{\alpha}+ \max\left\{\alpha, 1 \right\}Z_{\alpha}\left(\lambda\right)\right)\\
        & +\frac{1}{u}\left\|\hat{\Lambda}_{n}q_{m}\left(\hat{\Lambda}_{n}\right)T^{\frac{1}{2}}\hat{R}-T^{\frac{1}{2}}\hat{R}\right\|_{L^2},
    \end{split}
\end{equation*}
where $\tilde{u}= \min\left\{u, \left|p_{m}^{'}\left(0\right)\right|^{-1}\right\}$. Now we define our stopping rule as
\begin{equation*}
    m^* = \inf \left\{m \geq  0  \quad : \left\|\hat{\Lambda}_{n}q_{m}\left(\hat{\Lambda}_{n}\right)T^{\frac{1}{2}}\hat{R}-T^{\frac{1}{2}}\hat{R}\right\|_{L^2} \leq \Omega \right\},
\end{equation*}
Then
\begin{equation}
\label{errorbound2}
    \begin{split}
        \left\| \hat{\beta}_{m^*} - \beta^* \right\|_{L^2} \lesssim & \sqrt{\frac{ \sigma^2 \mathcal{N}\left(\lambda\right)}{n \delta}}\frac{\left(\tilde{u}+\lambda\right)^{\frac{1}{2}}}{\tilde{u}}+ \left({u}^{\alpha}+ \max\left\{\alpha, 1 \right\}Z_{\alpha}\left(\lambda\right)\right) +\frac{1}{u} \Omega.
    \end{split}
\end{equation}
\noindent
We still have to bound $\left|p_{m^*}^{'}\left(0\right)\right|$ and for that we will use Lemma \ref{Residual bound}. First we will bound $\left|p_{m^*-1}^{'}\left(0\right)\right|$. We assume that $0< u< x_{1,m-1}\leq x_{1,m-1}^{\left(2\right)}$ (see Lemma \ref{polynomial lemma}). Consider
\begin{equation*}
    \begin{split}
        &\left[p_{m-1},p_{m-1}\right]_{\left(0\right)}^{\frac{1}{2}}\\ =& \left\|p_{m-1}\left(\hat{\Lambda}_{n}\right)T^{\frac{1}{2}}\hat{R}\right\|_{L^2} \\
        \leq & \left\|p_{m-1}^{\left(2\right)}\left(\hat{\Lambda}_{n}\right)T^{\frac{1}{2}}\hat{R}\right\|_{L^2}  \quad \left( \text{by Lemma } \ref{polynomial lemma} \right)\\
        \leq & \left\|F_{u}p_{m-1}^{\left(2\right)}\left(\hat{\Lambda}_{n}\right)T^{\frac{1}{2}}\hat{R}\right\|_{L^2}+\left\|F_{u}^{\perp}p_{m-1}^{\left(2\right)}\left(\hat{\Lambda}_{n}\right)T^{\frac{1}{2}}\hat{R}\right\|_{L^2}\\
        \leq & \left\|F_{u}T^{\frac{1}{2}}\hat{R}\right\|_{L^2} + u^{-\frac{1}{2}}\left\|p_{m-1}^{\left(2\right)}\left(\hat{\Lambda}_{n}\right)\hat{\Lambda}_{n}^{\frac{1}{2}}T^{\frac{1}{2}}\hat{R}\right\|_{L^2} \\
        \leq & \left\|F_{u}\left(T^{\frac{1}{2}}\hat{R}-T^{\frac{1}{2}}\hat{C}_{n}\beta^*\right)\right\|_{L^2} + \left\|F_{u}T^{\frac{1}{2}}\hat{C}_{n}\beta^*\right\|_{L^2}+ u^{-\frac{1}{2}}\left\|p_{m-1}^{\left(2\right)}\left(\hat{\Lambda}_{n}\right)\hat{\Lambda}_{n}^{\frac{1}{2}}T^{\frac{1}{2}}\hat{R}\right\|_{L^2}.
             \end{split}
    \end{equation*}
    From Lemma \ref{Covariance estimation}, Lemma \ref{Empirical bound} and Lemma \ref{fb}, we get
    \begin{equation*}
        \begin{split}
        &\left[p_{m-1},p_{m-1}\right]_{\left(0\right)}^{\frac{1}{2}}\\ = & \sqrt{\frac{2 \sigma^2 \mathcal{N}\left(\lambda\right)}{n \delta}}\left\|F_{u}\left(\hat{\Lambda}_{n}+\lambda I\right)^{\frac{1}{2}}\right\|_{op} + \left\|F_{u}\hat{\Lambda}_{n}\Lambda^{\alpha}g\right\|_{L^2}+u^{-\frac{1}{2}}\left\|p_{m-1}^{\left(2\right)}\left(\hat{\Lambda}_{n}\right)\hat{\Lambda}_{n}^{\frac{1}{2}}T^{\frac{1}{2}}\hat{R}\right\|_{L^2} \\
        \leq & \sqrt{\frac{ 2\sigma^2 \mathcal{N}\left(\lambda\right)}{n \delta}}\left(u + \lambda\right)^{\frac{1}{2}}+2 c\left(\alpha \right) u \left(u^{\alpha}+Z_{\alpha}\left(\lambda\right)\right)\left\|g\right\|_{L^2}+u^{-\frac{1}{2}}\left\|p_{m-1}^{\left(2\right)}\left(\hat{\Lambda}_{n}\right)\hat{\Lambda}_{n}^{\frac{1}{2}}T^{\frac{1}{2}}\hat{R}\right\|_{L^2} \\
        = & \sqrt{\frac{ 2\sigma^2 \mathcal{N}\left(\lambda\right)}{n \delta}}\left(u + \lambda\right)^{\frac{1}{2}}+2 c\left(\alpha \right) u \left(u^{\alpha}+Z_{\alpha}\left(\lambda\right)\right)\left\|g\right\|_{L^2}+u^{-\frac{1}{2}}\left[p_{m-1}^{\left(2\right)},p_{m-1}^{\left(2\right)}\right]_{\left(1\right)}^{\frac{1}{2}},
    \end{split}
\end{equation*}
where $c\left(\alpha \right)$ is a constant depending on $\alpha$.
Here we have used that $\left|p_{m-1}^{\left(2\right)}\left(x\right)\right| \leq 1$  for $x \in \left[0,x_{m-1}^{\left(2\right)}\right]$. Using Assumption $3$ and with the choice of $\lambda = c\left(\alpha,\delta\right) n^{-\frac{1}{1+s+2 \alpha}}$, we get that $\sqrt{\frac{ \sigma^2 \mathcal{N}\left(\lambda\right)}{n \delta}} \lesssim \lambda^{\alpha+\frac{1}{2}}$ and $Z_{\alpha}\left(\lambda\right) \leq \lambda^{\alpha}$.\\
\noindent
From Lemma \ref{Residual bound} and the stopping rule $\left\|\hat{\Lambda}_{n}q_{m^*-1}\left(\hat{\Lambda}_{n}\right)T^{\frac{1}{2}}\hat{R}-T^{\frac{1}{2}}\hat{R}\right\|_{L^2} > \Omega$, we get
\begin{equation*}
    \begin{split}
        &\left(2+\tau\right)\lambda^{\frac{1}{2}}\lambda^{\frac{1}{2}+\alpha}  \\
        < &  \sqrt{\frac{2 \sigma^2 \mathcal{N}\left(\lambda\right)}{n \delta}}\left(\left|p_{m^*-1}^{'}\left(0\right)\right|^{-\frac{1}{2}}+\lambda^{\frac{1}{2}}\right)
        \\
        & + 2 \left(\left|p_{m^*-1}^{'}\left(0\right)\right|^{-\left(\alpha +1\right)}+c\left(\alpha \right)Z_{\alpha}\left(\lambda\right)\left|p_{m^*-1}^{'}\left(0\right)\right|^{-1}\right)\left\|g\right\|_{L^2}\\
         \leq & \sqrt{2}\lambda^{\frac{1}{2}+\alpha}\left(\left|p_{m^*-1}^{'}\left(0\right)\right|^{-\frac{1}{2}}+\lambda^{\frac{1}{2}}\right) \\
         & + 2 \left(\left|p_{m^*-1}^{'}\left(0\right)\right|^{-\left(\alpha +1\right)}+c\left(\alpha \right)Z_{\alpha}\left(\lambda\right)\left|p_{m^*-1}^{'}\left(0\right)\right|^{-1}\right)\left\|g\right\|_{L^2}.
    \end{split}
\end{equation*}
Therefore, we have
         \begin{equation*} 
         \begin{split}
         \tau \lambda^{\frac{1}{2}}\lambda^{\frac{1}{2}+\alpha} \leq  c\left(\alpha \right) \max &  \left\{\lambda^{\frac{1}{2}}\lambda^{\frac{1}{2}+\alpha}\left|p_{m^*-1}^{'}\left(0\right)\right|^{-\frac{1}{2}},\left|p_{m^*-1}^{'}\left(0\right)\right|^{-\left(\alpha +1\right)}\right.\left\|g\right\|_{L^2},\\        & \left. Z_{\alpha}\left(\lambda\right)\left|p_{m^*-1}^{'}\left(0\right)\right|^{-1}\left\|g\right\|_{L^2}\right\}.
        \end{split}
\end{equation*}
From all three cases, we will get that
\begin{equation*}
    \left|p_{m^*-1}^{'}\left(0\right)\right| \leq {c_2}\left(\alpha, \tau \right) \lambda^{-1}.
\end{equation*}
In the next step we get an bound on $\left|p_{m^*}^{'}\left(0\right)\right|$. From Lemma \ref{polynomial lemma}, we know that
\begin{equation*}
    \left|p_{m-1}^{'}\left(0\right)-p_{m}^{'}\left(0\right)\right| \leq \frac{\left[p_{m-1},p_{m-1}\right]_{\left(0\right)}}{\left[p_{m-1}^{\left(2\right)},p_{m-1}^{\left(2\right)}\right]_{\left(1\right)}}.
\end{equation*}
Define $u : = a\left(\alpha, \tau\right) \lambda$, where $a\left(\alpha, \tau\right)$  is a constant. Using this choice we get
\begin{equation*}
    \begin{split}
        \sqrt{\frac{2 \sigma^2 \mathcal{N}\left(\lambda\right)}{n \delta}}\left(u + \lambda\right)^{\frac{1}{2}}+2 c\left(\alpha \right) u \left(u^{\alpha}+Z_{\alpha}\left(\lambda\right)\right)\left\|g\right\|_{L^2} \leq  c_3\left(\alpha,\tau\right)\lambda^{\frac{1}{2}}\lambda^{\alpha+\frac{1}{2}}.
    \end{split}
\end{equation*}
From the stopping rule, we know that
\begin{equation*}
\left\|\hat{\Lambda}_{n}q_{m^*-1}\left(\hat{\Lambda}_{n}\right)T^{\frac{1}{2}}\hat{R}-T^{\frac{1}{2}}\hat{R}\right\|_{L^2} > \Omega = \left(2+\tau\right)\lambda^{\frac{1}{2}}\lambda^{\alpha+\frac{1}{2}}.
\end{equation*}
So now we can see that using all these inequalities we have that 
\begin{equation}
\label{pmbound}
    \left|p_{m^*}^{'}\left(0\right)\right| \leq c_4\left(\alpha,\tau\right)\lambda^{-1}.
\end{equation}
Using $\left(\ref{pmbound}\right)$, the choice of $a\left(\alpha,\tau\right)$ can be made accordingly such that
\begin{equation*}
    u \leq \left|p_{m^*}^{'}\left(0\right)\right|^{-1} \leq x_{1,m}.
\end{equation*}
With this inequality, we can choose $\tilde{u}=u$.
Now with this choice of $\tilde{u}= a\left(\alpha,\tau\right)\lambda $, where $a\left(\alpha,\tau\right)$ has been taken to satisfy all the conditions on $\tilde{u}$, we will further bound $\left(\ref{errorbound2}\right)$. Therefore, we get that
\begin{equation*}
    \begin{split}
        \left\| \hat{\beta}_{m^*} - \beta^* \right\|_{L^2} \lesssim  & \sqrt{2} \delta\left(\lambda \right) \tilde{u}^{-1}\left(\lambda+\tilde{u}\right)^{\frac{1}{2}}+2 \left({u}^{\alpha}+ \max\left\{\alpha, 1 \right\}Z_{\alpha}\left(\lambda\right)\right)\\
        & +\frac{1}{u}\left\|\hat{\Lambda}_{n}q_{m^*}\left(\hat{\Lambda}_{n}\right)T^{\frac{1}{2}}\hat{R}-T^{\frac{1}{2}}\hat{R}\right\|_{L^2}\le
        \lambda^{\alpha}.
    \end{split}
\end{equation*}
Now by using $\lambda =c\left(\alpha,\delta\right) n^{-\frac{\alpha +1}{1+s+2\alpha}}$, the result follows.
\end{proof}
In the RKHS framework, the first minimax convergence rate was established by
Cai and Yuan \cite{ARKHSFORFLR,tonyyuan2012minimax} for the Tikhonov regularization method with the source condition $\beta^* \in \mathcal{H}$. Later in \cite{ZhangFaster2020}, the results were extended, and the minimax convergence rates  derived with the source condition 
$\beta^* \in \mathcal{R}\left(T^{\frac{1}{2}}\left(T^{\frac{1}{2}}CT^{\frac{1}{2}}\right)^{\alpha}\right) \text{ for } 0< \alpha \leq \frac{1}{2} $.
 Our convergence rates for the conjugate gradient method match the minimax rates of \cite{ZhangFaster2020} 
 and also match with the convergence of rates of \cite{guo2023capacity}. 

\section{Supplementary Results}
\label{supplements}
Technical details in the paper depend on the estimation of the residual term as the CG method works to reduce the residual term as the number of iterations progresses. To simplify the technical part we use the fact that $\left\|J^{*} \hat{C}_{n}J \hat{\beta}_m-J^{*}\hat{R}\right\|_{\mathcal{H}} = \left\|\hat{\Lambda}_{n}q_{m}\left(\hat{\Lambda}_{n}\right)T^{\frac{1}{2}}\hat{R}-T^{\frac{1}{2}}\hat{R}\right\|_{L^2}$. We introduce several lemmas that aid us to estimate the residual term and to prove the main theorem.
\begin{lemma}
\label{Covariance estimation}
Under Assumptions $(2)$ and $(3)$, we get
\begin{equation*} 
\left\|\left(\hat{\Lambda}_{n}-\Lambda\right)\left(\Lambda+\lambda I\right)^{-1}\right\|_{op} \leq c_{1}\left(n \lambda^{1+s}\right)^{-\frac{1}{2}}
\end{equation*}
for some constant $c_{1}>0$.
\end{lemma}
\noindent
We skip the proof of this lemma as it follows from the similar steps of Lemma 2 in \cite{tonyyuan2012minimax}.\\
From Lemma \ref{Covariance estimation} and for $\lambda \geq \left(\frac{4c_{1}^2}{n}\right)^{\frac{1}{s+1}}$, we get
\begin{equation*}
    \left\|\left(\hat{\Lambda}_{n}-\Lambda\right)\left(\Lambda+\lambda I\right)^{-1}\right\|_{op} \leq \frac{1}{2}.
\end{equation*}
As a consequence, we get
\begin{equation*}
    \begin{split}
        \left\|\left(\Lambda+\lambda I\right)\left(\hat{\Lambda}_{n}+\lambda I\right)^{-1}\right\|_{op} 
        & = \left\|\left[\left(\hat{\Lambda}_{n}-\Lambda\right)\left(\Lambda+\lambda I\right)^{-1}+I\right]^{-1}\right\|_{op} \\
        & \leq \frac{1}{1-\left\|\left(\hat{\Lambda}_{n}-\Lambda\right)\left(\Lambda+\lambda I\right)^{-1}\right\|_{op}}\\
        & \leq 2, \quad \forall  \lambda \geq \left(\frac{4c_{1}^2}{n}\right)^{\frac{1}{s+1}}.
    \end{split}
\end{equation*}
Using Corde's inequality, $\left(\left\|A^{\nu}B^{\nu}\right\|_{op}\leq \left\|AB\right\|_{op}^{\nu},\,0\le \nu\le 1\right)$, where $A$ and $B$ are self-adjoint positive operators, we get
\begin{equation}
\label{covariance estimation1}
    \left\|\left(\Lambda+\lambda I\right)^{\nu}\left(\hat{\Lambda}_{n}+\lambda I\right)^{-\nu}\right\|_{op} \leq 2^{\nu}, \quad \forall  \nu \in S, \lambda \geq \left(\frac{4c_{1}^2}{n}\right)^{\frac{1}{s+1}}.
\end{equation}
\begin{lemma}
\label{Empirical bound}
For $\delta >0$, with at least probability $1-\delta$, we have that
\begin{equation*}
\left\|\left(\Lambda+\lambda I\right)^{-\frac{1}{2}}T^{\frac{1}{2}}\left(\hat{R}-\hat{C}_{n}\beta^*\right)\right\|_{L^2} \leq \sqrt{\frac{ \sigma^2 \mathcal{N}\left(\lambda\right)}{n \delta}}.
\end{equation*}

\end{lemma}
\begin{proof}
Define $Z_{i}:= \left(\Lambda+\lambda I\right)^{-\frac{1}{2}}T^{\frac{1}{2}}\left[Y_{i}X_{i}-\left(X_{i}\otimes X_{i}\right)\beta^*\right]$. Since the slope function $\beta^*$ satisfies the operator equation $C\beta^*= \mathbb{E}\left[YX\right]$, we get that the mean of random variable $Z_{i}$ is zero, i.e.,
$$\mathbb{E}\left[Z_{i}\right]=\left(\Lambda+\lambda I\right)^{-\frac{1}{2}}T^{\frac{1}{2}}\left[\mathbb{E}\left[YX\right]-C\beta^*\right] =0.$$
By Markov's inequality, for any $t>0$
$$ \mathbb{P}\left(\left\|\frac{1}{n}\sum_{i=1}^{n}Z_{i}\right\|_{L^2} \geq t\right)\leq \frac{\mathbb{E}\left\|\frac{1}{n}\sum_{i=1}^{n}Z_{i}\right\|_{L^2}^{2}}{t^2}.$$
Note that 
$$\mathbb{E}\left\|\frac{1}{n}\sum_{i=1}^{n}Z_{i}\right\|_{L^2}^{2} = \frac{1}{n^2}\sum_{i,j=1}^{n}\mathbb{E}\left\langle Z_{i},Z_{j}\right\rangle_{L^2}= \frac{1}{n^2}\sum_{i\neq j}^{n}\mathbb{E}\left\langle Z_{i},Z_{j}\right\rangle_{L^2}+ \frac{1}{n^2}\sum_{i=1}^{n}\mathbb{E}\left\|Z_{i}\right\|_{L^2}^{2}= \frac{\mathbb{E}\left\|Z_{1}\right\|_{L^2}^{2}}{n} $$
and by taking $t = \sqrt{\frac{\mathbb{E}\left\|Z_{1}\right\|_{L^2}^{2}}{n\delta}}$, with at least probability $1-\delta$, we have 
\begin{equation}
\label{chebyshevbound}
\left\|\left(\Lambda+\lambda I\right)^{-\frac{1}{2}}T^{\frac{1}{2}}\left(\hat{R}-\hat{C}_{n}\beta^*\right)\right\|_{L^2} \leq \sqrt{\frac{\mathbb{E}\left[\left\|\left(\Lambda+\lambda I\right)^{-\frac{1}{2}}T^{\frac{1}{2}}\left(YX-\left(X \otimes X\right)\beta^*\right)\right\|_{L^2}^{2}\right]}{n\delta}}.
\end{equation}

Consider
\begin{equation*}
    \begin{split}
        & \mathbb{E}\left[\left\|\left(\Lambda+\lambda I\right)^{-\frac{1}{2}}T^{\frac{1}{2}}\left(YX-\left(X \otimes X\right)\beta^*\right)\right\|_{L^2}^{2}\right]\\
        = & \mathbb{E}\left[\left\|\left(\Lambda+\lambda I\right)^{-\frac{1}{2}}T^{\frac{1}{2}}\left(Y-\left\langle X, \beta^* \right\rangle_{L^2}\right)X\right\|_{L^2}^{2}\right]\\
        = & \mathbb{E}\left[\left(Y-\left\langle X, \beta^* \right\rangle_{L^2}\right)^2\left\|\left(\Lambda+\lambda I\right)^{-\frac{1}{2}}T^{\frac{1}{2}}X\right\|_{L^2}^{2}\right]\\
        = & \mathbb{E}\left[\epsilon^{2} \left\langle \left(\Lambda+\lambda I\right)^{-\frac{1}{2}}T^{\frac{1}{2}}X, \left(\Lambda+\lambda I\right)^{-\frac{1}{2}}T^{\frac{1}{2}}X\right\rangle_{L^2}\right]\\
        = & \mathbb{E}\left[\epsilon^{2}\text{trace}\left(\left(\Lambda+\lambda I\right)^{-1}T^{\frac{1}{2}}\left(X \otimes X\right)T^{\frac{1}{2}}\right)\right] \\
        = & \mathbb{E}\left[\epsilon^{2} \right] \text{trace}\left(\left(\Lambda+\lambda I\right)^{-1}T^{\frac{1}{2}}CT^{\frac{1}{2}}\right) \\
        = & \sigma^2 \text{trace}\left(\left(\Lambda+\lambda I\right)^{-1}\Lambda\right) = \sigma^2 \mathcal{N}\left(\lambda\right).
    \end{split}
\end{equation*}
With this bound and $\left(\ref{chebyshevbound}\right)$, the result follows.
\end{proof}
\noindent
\begin{lemma}
\label{HSBound}
    Assuming $\mathbb{E}\left\|X\right\|^4 < \infty$, with at least probability $1-\delta$, we get
    \begin{equation*}
        \left\|\hat{C}_{n}-C\right\|_{HS} \leq \sqrt{\frac{\mathbb{E}\left\|X\right\|^4}{n \delta}} := \Delta.
    \end{equation*}
\begin{proof}
    From Chebyshev's inequality, we have that
    \begin{equation*}
        \mathbb{P}\left(\left\|\hat{C}_{n}-C\right\|_{HS} > \xi\right) \leq \frac{\mathbb{E}\left\|\hat{C}_{n}-C\right\|_{HS}^2}{\xi^2}.
    \end{equation*}
Using Theorem 2.5 \cite{hovarthkokoszka2012}, we get
\begin{equation*}
        \mathbb{P}\left(\left\|\hat{C}_{n}-C\right\|_{HS} > \xi\right) \leq \frac{\mathbb{E}\left\|X\right\|^4}{n \xi^2}.
    \end{equation*}
Taking $\xi = \sqrt{\frac{\mathbb{E}\left\|X\right\|^4}{n \delta}}$ will conclude the result.
\end{proof}
\end{lemma}
\noindent
The bound in the Hilbert-Schmidt is stronger than the operator norm. So the above lemma provides a stronger estimation compared to the estimation in terms of operator norm. The next lemma explains a technical bound involving operator $\Lambda$ and $\hat{\Lambda}_{n}$ which will be used repeatedly in our analysis.
\begin{lemma}
\label{fb}
For any $\nu > 0$, and measurable $\phi: \left[0, \kappa\right] \to \mathbb{R}$, it holds with probability greater than  $1-e^{\xi}$ that
\begin{equation*}
\left\|\phi\left(\hat{\Lambda}_{n}\right)\Lambda^{\nu}\right\|_{op} \lesssim  \sup_{t \in \left[0,\kappa\right]}t^{\nu}\phi\left(t\right) + \max\left\{\nu,1\right\} Z_{\nu}\left(\lambda\right)\sup_{t \in \left[0,\kappa\right]}\phi\left(t\right),
\end{equation*}
where
\begin{equation*}
        Z_{\nu}\left(\lambda\right) = 
        \begin{cases}
            \lambda^{\nu}, & \mbox{ if } \nu \leq 1\\
         \kappa^{\nu} \Delta, & \mbox{ if } \nu > 1 
        \end{cases}.
    \end{equation*}
\end{lemma}
\begin{proof}
Proof of this result follows the similar steps of Lemma 5.3 \cite{BlanchardCGKCG2016}.
For $\nu \leq 1$, we have
\begin{equation*}
    \begin{split}
        \left\|\phi\left(\hat{\Lambda}_{n}\right)\Lambda^{\nu}\right\|_{op} & \leq \left\|\phi\left(\hat{\Lambda}_{n}\right)\left(\hat{\Lambda}_{n}+\lambda I\right)^{\nu}\right\|_{op}\left\|\left(\hat{\Lambda}_{n}+\lambda I\right)^{-\nu}\left(\Lambda+\lambda I\right)^{\nu}\right\|_{op}\left\|\left(\Lambda+\lambda I\right)^{-\nu}\Lambda^{\nu}\right\|_{op}\\
        & \lesssim \left(\sup_{t \in \left[0,\kappa\right]}t^{\nu}\phi\left(t\right)+\lambda^{\nu}\sup_{t \in \left[0,\kappa\right]}\phi\left(t\right)\right).
    \end{split}
\end{equation*} 
For last inequality, we used $\left(\ref{covariance estimation1}\right)$ and the fact that $\left\|\left(\Lambda+\lambda I\right)^{-\nu}\Lambda^{\nu}\right\|_{op} \leq 1$.  
For $\nu >1$,
\begin{equation*}
    \begin{split}
        \left\|\phi\left(\hat{\Lambda}_{n}\right)\Lambda^{\nu}\right\|_{op} & \leq \left\|\phi\left(\hat{\Lambda}_{n}\right)\right\|_{op}\left\|\left(\Lambda^{\nu}-\hat{\Lambda}_{n}^{\nu}\right)\right\|_{op}+\left\|\phi\left(\hat{\Lambda}_{n}\right)\hat{\Lambda}_{n}^{\nu}\right\|_{op} \\
        & \leq \left\|\left(\Lambda^{\nu}-\hat{\Lambda}_{n}^{\nu}\right)\right\|_{op} \sup_{t \in \left[0,\kappa\right]}\phi\left(t\right) + \sup_{t \in \left[0,\kappa\right]}t^{\nu}\phi\left(t\right) \\
        & \lesssim \left\|\Lambda-\hat{\Lambda}_{n}\right\|_{HS}\sup_{t \in \left[0,\kappa\right]}\phi\left(t\right) + \sup_{t \in \left[0,\kappa\right]}t^{\nu}\phi\left(t\right).
    \end{split}
\end{equation*}
Here we used that for $\nu >1$, $x^{\nu}$ is $\nu \kappa^{\nu-1}-$ Lipschitz over $\left[0,\kappa\right] $ and by Lemma \ref{HSBound} we get our result.
\end{proof}
\noindent
In the following lemma, we will discuss the bound of the residual term that will be used later to bound $\left|p_{m}^{'}\left(0\right)\right|$.
\begin{lemma}
\label{Residual bound}
    Under Assumptions $\left(1\right)$--$\left(3\right)$, $\mathbb{E}\left\|X\right\|^4 < \infty$ and $\lambda \geq \left(\frac{4c_{1}^2}{n}\right)^{\frac{1}{s+1}}$, we have that
    \begin{equation*}
    \begin{split}
\left\|\hat{\Lambda}_{n}q_{m}\left(\hat{\Lambda}_{n}\right)T^{\frac{1}{2}}\hat{R}-T^{\frac{1}{2}}\hat{R}\right\|_{L^2} 
        \leq &  \sqrt{\frac{2 \sigma^2 \mathcal{N}\left(\lambda\right)}{n \delta}}\left(\left|p_{m}^{'}\left(0\right)\right|^{-\frac{1}{2}}+\lambda^{\frac{1}{2}}\right)\\
        & + 2 \left(\left|p_{m}^{'}\left(0\right)\right|^{-\left(\alpha +1\right)}+c\left(\alpha \right)Z_{\alpha}\left(\lambda\right)\left|p_{m}^{'}\left(0\right)\right|^{-1}\right)\left\|g\right\|_{L^2}.
    \end{split}
\end{equation*}
\end{lemma}
\begin{proof}
For bounding the residual term, we use Lemma \ref{polynomial lemma} at the initial stage to conclude that
    \begin{equation*}
    \begin{split}
       & \left\|\hat{\Lambda}_{n}q_{m}\left(\hat{\Lambda}_{n}\right)T^{\frac{1}{2}}\hat{R}-T^{\frac{1}{2}}\hat{R}\right\|_{L^2}\\
       = & \left\|p_{m}\left(\hat{\Lambda}_{n}\right)T^{\frac{1}{2}}\hat{R}\right\|_{L^2} \\
       \leq & \left\|F_{x_{1,m}}\left(\phi_{m}\left(\hat{\Lambda}_{n}\right)T^{\frac{1}{2}}\hat{R}\right)\right\|_{L^2} \\
       = & \left\|F_{x_{1,m}}\left(\phi_{m}\left(\hat{\Lambda}_{n}\right)\left(T^{\frac{1}{2}}\hat{R}-T^{\frac{1}{2}}\hat{C}_{n}\beta^*+T^{\frac{1}{2}}\hat{C}_{n}\beta^*\right)\right)\right\|_{L^2} \\
       \leq & \underbrace{\left\|F_{x_{1,m}}\left(\phi_{m}\left(\hat{\Lambda}_{n}\right)T^{\frac{1}{2}}\left(\hat{R}-\hat{C}_{n}\beta^*\right)\right)\right\|_{L^2}}_{Term-A} + \underbrace{\left\|F_{x_{1,m}}\left(\phi_{m}\left(\hat{\Lambda}_{n}\right)T^{\frac{1}{2}}\hat{C}_{n}\beta^*\right)\right\|_{L^2}}_{Term-B}.
    \end{split}
\end{equation*}
 \noindent
 We will take help from Lemma \ref{Covariance estimation} and Lemma \ref{Empirical bound} for the estimation of both terms.\\

\vspace{0.07cm}

\textit{Estimation of Term-A: }
\begin{equation}
\label{terma}
    \begin{split}
        & \left\|F_{x_{1,m}}\left(\phi_{m}\left(\hat{\Lambda}_{n}\right)T^{\frac{1}{2}}\left(\hat{R}-\hat{C}_{n}\beta^*\right)\right)\right\|_{L^2} \\
        \leq & \left\|F_{x_{1,m}}\phi_{m}\left(\hat{\Lambda}_{n}\right)\left(\hat{\Lambda}_{n}+\lambda I\right)^{\frac{1}{2}}\right\|_{op}\\
        &\qquad\qquad\times\left\|\left(\hat{\Lambda}_{n}+\lambda I\right)^{-\frac{1}{2}}\left(\Lambda+\lambda I\right)^{\frac{1}{2}}\right\|_{op}\left\|\left(\Lambda+\lambda I\right)^{-\frac{1}{2}}\left(T^{\frac{1}{2}}\hat{R}-T^{\frac{1}{2}}\hat{C}_{n}\beta^*\right)\right\|_{L^2}\\
        \leq & \sqrt{\frac{2 \sigma^2 \mathcal{N}\left(\lambda\right)}{n \delta}}\left\|F_{x_{1,m}}\phi_{m}\left(\hat{\Lambda}_{n}\right)\left(\hat{\Lambda}_{n}+\lambda I\right)^{\frac{1}{2}}\right\|_{op}\\
        \leq &  \sqrt{\frac{2 \sigma^2 \mathcal{N}\left(\lambda\right)}{n \delta}}\left(\sup_{x \in \left[0,x_{1,m}\right]}x^{\frac{1}{2}}\phi_{m}\left(x\right) + \lambda^{\frac{1}{2}}\sup_{x \in \left[0,x_{1,m}\right]}\phi_{m}\left(x\right)\right)\\
        \leq &  \sqrt{\frac{2 \sigma^2 \mathcal{N}\left(\lambda\right)}{n \delta}}\left(\left|p_{m}^{'}\left(0\right)\right|^{-\frac{1}{2}}+\lambda^{\frac{1}{2}}\right).
    \end{split}
\end{equation}
For the last inequality, we use $\left(\ref{xphi bound}\right)$ for $\nu= 0,1$.\\
\noindent
\vspace{0.07cm}

\textit{Estimation of Term-B:} 
\begin{equation}
\label{termb}
    \begin{split}
        & \left\|F_{x_{1,m}}\left(\phi_{m}\left(\hat{\Lambda}_{n}\right)T^{\frac{1}{2}}\hat{C}_{n}\beta^*\right)\right\|_{L^2} \\
        = & \left\|F_{x_{1,m}}\phi_{m}\left(\hat{\Lambda}_{n}\right)\hat{\Lambda}_{n}\Lambda^{\alpha}\right\|_{op}\left\|g\right\|_{L^2} \\
        \leq & 2\left(\sup_{t \in \left[0,x_{1,m}\right]}t^{\alpha+1}\phi_{m}\left(t\right)+ c\left(\alpha \right)Z_{\alpha}\left(\lambda\right) \sup_{t \in \left[0,x_{1,m}\right]}t\phi_{m}\left(t\right)\right)\left\|g\right\|_{L^2}\\
        \leq & 2 \left(\left|p_{m}^{'}\left(0\right)\right|^{-\left(\alpha +1\right)}+c\left(\alpha \right)Z_{\alpha}\left(\lambda\right)\left|p_{m}^{'}\left(0\right)\right|^{-1}\right)\left\|g\right\|_{L^2}.
    \end{split}
\end{equation}
where we used Lemma \ref{fb} and $\left(\ref{xphi bound}\right)$ in the last inequality. From $\left(\ref{terma}\right)$ and $\left(\ref{termb}\right)$, we obtain the result.
\end{proof}
\noindent

\section*{Acknowledgments}
S. Sivananthan acknowledges the Science and Engineering Research Board, Government of India, for the financial support through project no. MTR/2022/000383. BKS is partially supported by the National Science Foundation DMS CAREER Award 1945396.
\bibliographystyle{acm}
\bibliography{main}
\end{document}